\documentclass[12pt,a4paper]{amsart}
\usepackage{latexsym}
\usepackage{amssymb}
\usepackage[bookmarks=true,hyperindex,pdftex,colorlinks,citecolor=blue]{hyperref}
 \newtheorem{theo}{Theorem}[section]

  \newtheorem{lem}[theo]{Lemma}
  \newtheorem{prob}[theo]{Problem}
    \newtheorem{corollary}[theo]{Corollary}
    \newtheorem{prop}[theo]{Proposition}

  \newcommand{\dopu}{{:}\allowbreak\ }

\newcommand{\R}{{\mathbb{R}}}

\newcommand{\N}{{\mathbb{N}}}

\newcommand{\loglike}[1]{\mathop{\rm #1}\nolimits}
\newcommand{\ex}{\loglike{ext}}

\newcommand{\aconv}{{\mathrm{aconv}}}
\newcommand{\spn}{\mathrm{span}}

\newcommand{\bea}{\begin{eqnarray*}}
\newcommand{\eea}{\end{eqnarray*}}
\newcommand{\beq}{\begin{equation}}
\newcommand{\eeq}{\end{equation}}

\numberwithin{equation}{section}

\begin{document}
A technical version with some boring proofs included. The main article has already been published: Zavarzina O. \textit{Non-expansive bijections between unit balls of Banach spaces},   Annals of Functional Analysis, \textbf{9}, Number 2 (2018), 271--281. This longer
version posted for the convenience of readers.
\title[Non-expansive bijections between unit balls]{Non-expansive bijections between unit balls of Banach spaces}



\author{Olesia Zavarzina}

\address{Department of Mathematics and Informatics\\ V.N. Karazin Kharkiv  National University\\ 61022 Kharkiv, Ukraine}
\email{olesia.zavarzina@yahoo.com}

\begin{abstract}
It is known that if  $M$ is a finite-dimensional Banach space, or a strictly convex  space, or the space $\ell_1$, then every non-expansive bijection $F: B_M \to B_M$ is an isometry. We extend these results to non-expansive bijections $F: B_E \to B_M$ between unit balls of two different Banach spaces. Namely, if $E$ is an arbitrary Banach space and $M$ is finite-dimensional or strictly convex, or the space $\ell_1$ then every non-expansive bijection  $F: B_E \to B_M$ is an isometry.
\end{abstract}
\maketitle


\section{Introduction}
 Let $M$ be a metric space. A map $F: M \to M$  is called \textit{non-expansive}, if $\rho(F(x), F(y)) \le \rho(x,y)$ for all $x,y \in M$. The space $M$ is called \textit{expand-contract  plastic} (or simply, an EC-space) if every non-expansive bijection from $M$ onto  itself is an isometry.

 It is known \cite[Theorem 1.1]{NaiPioWing} that every compact (or even totally bounded) metric space is expand-contract  plastic, so in particular every bounded subset of  $\R^n$ is an EC-space.

 The situation with bounded subsets of  infinite dimensional spaces is different. On the one hand, there is  a  non-expand-contract  plastic bounded closed convex subset of a Hilbert space \cite[Example 2.7] {CKOW2016} (in fact, that set is an ellipsoid), but on the other hand, the unit ball of a Hilbert space, and in general  the unit ball of every strictly convex Banach space is an EC-space  \cite[Theorem 2.6] {CKOW2016}. It is unknown whether the strict convexity condition in  \cite[Theorem 2.6] {CKOW2016} can be omitted, that is, in other words, the following problem arises.

 \begin{prob} \label{problem1}
For what Banach spaces $Y$ every bijective non-expansive map $F: B_Y \to B_Y$  is an isometry?  Is this true for every Banach space?
 \end{prob}

Outside of strictly convex spaces, Problem \ref{problem1} is solved positively for all finite-dimensional spaces (because of the compactness of the unit ball), and for the space $\ell_1$  \cite[Theorem 1]{KZ}.

To the best of our knowledge, the following natural extension of  Problem \ref{problem1} is also open.

 \begin{prob} \label{problem2}
 Is it true that for every pair $(X, Y)$ of Banach spaces,  every bijective non-expansive map $F: B_X \to B_Y$  is an isometry?
 \end{prob}

 An evident bridge between these two problems is the third one, which we also are not able to solve.

 \begin{prob} \label{problem3}
Let $X, Y$ be Banach spaces that admit a bijective non-expansive map $F: B_X \to B_Y$. Is it true that spaces $X$ and $Y$ are isometric?
 \end{prob}

 In fact, if one solves Problem \ref{problem2} in positive, one evidently solves also Problem \ref{problem3}. On the other hand, for a fixed pair  $(X, Y)$  the positive answers to Problems \ref{problem1} and \ref{problem3}  would imply the same for Problem \ref{problem2}.
 \vspace{2mm}

The aim of this article is to demonstrate that for all spaces $Y$ where Problem \ref{problem1} is known to have a positive solution (i.e.  strictly convex spaces, $\ell_1$, and finite-dimensional spaces),  Problem \ref{problem2} can be solved in the  positive for all pairs of the form $(X, Y)$.  In fact, our result for pairs $(X,Y)$ with  $Y$ being strictly convex repeats the corresponding proof of the case $X=Y$ from \cite[Theorem 2.6] {CKOW2016} almost word-to-word. The proof for pairs  $(X,\ell_1)$ on some stage needs additional work comparing to its particular case $X = \ell_1$ from \cite[Theorem 1]{KZ}. The most difficult one is the finite-dimensional case, where the approach from  \cite[Theorem 1.1]{NaiPioWing} is not applicable for maps between two different spaces, because it uses iterations of the map. So, for finite-dimensional spaces we had to search for a completely different proof.  Our proof in this case uses some ideas from \cite{CKOW2016} and \cite{KZ} but elaborates them a lot.

 \vspace{3mm}

 There is another similar circle of problems that motivates our study.

   In 1987, D.~Tingley  \cite{ting} proposed the following question: let $f$ be a bijective isometry between the unit spheres $S_X$ and $S_E$ of real Banach spaces $X$, $E$
respectively. Is it true that $f$ extends to a linear
(bijective) isometry $F: X \to E$ of the
corresponding spaces?

Let us mention that this is equivalent to
the fact that the following natural positive-homogeneous extension $F: X \to E$ of $f$ is linear:
$$
F(0) = 0,\qquad F(x) = \|x\|\,f\left(x / \|x\|\right)\quad \bigl(x\in X\setminus\{0\}\bigr).
$$
Since according to P.~Mankiewicz's theorem \cite{mank} every bijective isometry between convex bodies can be uniquely extended to an affine isometry of the whole spaces,  Tingley's problem can be reformulated as follows:

 \begin{prob} \label{problem4}
 Let $F: B_X \to B_E$ be a positive-homogeneous map, whose restriction to $S_X$ is a bijective isometry between $S_X$ and $S_E$. Is it true that $F$ is an isometry itself?
 \end{prob}

There is a number of publications devoted to Tingley's problem
(see \cite{ding} for a survey of corresponding results) and, in particular, the problem is solved in the positive for many concrete classical Banach spaces. Surprisingly,  for general spaces this innocently-looking question remains open even in dimension two. For finite-dimensional polyhedral spaces the problem is solved in the positive by V.~Kadets and M.~Mart\'{\i}n in 2012 \cite{kad-mar}, and the positive solution for the class of generalized lush spaces was given by Dongni Tan, Xujian Huang, and Rui Liu  in 2013 \cite{TanHuangLiu}.  A step in the proof of the latter result was a lemma  (Proposition 3.4 of \cite{TanHuangLiu}) which in our terminology says that if the map $F$ in Problem \ref{problem4} is non-expansive, then the problem has a positive solution.  So, the problem which we address in our paper (Problem \ref{problem2}) can be considered as a much stronger variant of that lemma.


 \section{Preliminaries}
  In the sequel, the letters $X$ and $Y$ always stand for real Banach spaces. We denote by $S_X$ and $B_X$ the unit sphere and the closed unit ball of $X$ respectively. For a convex set $A \subset X$ denote by $\ex(A)$ the set of extreme points of $A$; that is, $x \in \ex(A)$ if $x \in A$ and for every $y \in X\setminus\{0\}$ either $x + y \not\in A$ or  $x - y \not\in A$. Recall that $X$ is called strictly convex if all elements of $S_X$ are extreme points of $B_X$, or in other words, $S_X$ does not contain non-trivial line segments. Strict convexity of $X$ is equivalent to the strict triangle inequality $\|x +y\| < \|x\| + \|y\|$ holding for all pairs of vectors $x, y \in X$ that do not have the same direction.
For subsets $A, B \subset X$ we use the standard notation $A+B = \{x + y\dopu x \in A, y \in B\}$ and $a A = \{ax \dopu x \in A\}$.\\
Now let us reformulate the results of \cite{CKOW2016} on the case of two different spaces. In order to do it we give the following lemmas from the same source.
For  $x \in S_X$ and $a \in (0,1)$ let $D(x, a):= a B_X \cap (x + (1 -a) B_X)$.
\begin{lem}[Lemma 2.1 of \cite{CKOW2016}] \label{subset}
 For every  $x \in S_X$ and $a \in (0,1)$
\beq \label{eqDxa}
 D(x, a) = a\{x + y \in B_X: x - \frac{a}{1-a} y \in B_X \}.
\eeq
When  $x$ is an extreme point of $B_X$, then $D(x, a) = \{ax\}$. When  $x$ is not an extreme point of $B_X$, then $D(x, \frac12)$ consists of more than one point.
\end{lem}

\begin{lem} [Lemma 2.2 of \cite{CKOW2016}]  \label{subsetD1}
 $D_1(x) = -x + \{x + y \in B_X: x -  y \in B_X \}$.
When  $x$ is an extreme point of $B_X$, then $D_1(x) = \{0\}$.
\end{lem}

The following theorem generalizes \cite[Theorem 2.3]{CKOW2016}, where the case $X=Y$ was considered. It can be demonstrated repeating the proof of \cite[Theorem 2.3]{CKOW2016} almost word to word.
\begin{theo} \label{wing-sc}
Let $F: B_X \to B_Y$ be a non-expansive bijection.  In the above notations the following hold.
\begin{enumerate}
\item  $F(0) = 0$.
\item  $F^{-1}(S_Y) \subset S_X$.
\item  $F(D(x, a)) \subset D(F(x), a)$  for all $x \in F^{-1}(S_Y)$ and $a \in (0,1)$.
\item  If $F(x)$ is an extreme point of $B_Y$, then $F(ax) = a F(x)$ for all $a \in (0,1)$.
\item  If $F(x)$ is an extreme point of $B_Y$, then $x$ is also an extreme point of $B_X$.
\item  If $F(x)$ is an extreme point of $B_Y$, then $F(-x) = -F(x)$.
\end{enumerate}
Moreover, if $Y$ is strictly convex, then
\begin{enumerate}
\item[(i)]  $F$ maps $S_X$ bijectively onto $S_Y$;
\item[(ii)]  $F(ax) = a F(x)$ for all $x \in S_X$ and $a \in (0,1)$;
\item[(iii)]  $F(-x) = -F(x)$ for all $x \in S_X$.
\end{enumerate}
\end{theo}

\begin{proof} (1) The point 0 has a special property: it is the unique point of the ball whose distance to all other points of the ball is less than or equal to 1. Evidently, a non-expansive bijection preserves this property.
\vspace{1 mm}

 (2) Notice that (1) implies that if $\|x\|_X < 1$, then  $\|F(x)\|_Y < 1$.Hence $F$ maps interior points into interior points.
So  $F^{-1}(S_Y) \subset S_X$.
\vspace{3 mm}

 (3) $D(x, a)$ is the set of those points $z$ such that $\|z\|_X \le a$ and $\|x -  z\|_X \le 1 - a$. So for every $z \in D(x, a)$,   $F(z)$ must satisfy the conditions $\|F(z)\|_Y \le a$ and $\|F(x) -  F(z)\|_Y \le 1 - a$, that is, $F(z) \in D(F(x), a)$.
  \vspace{3 mm}

(4) Let $F(x) \in \ex(B_Y)$. Then $D(F(x), a) = \{aF(x)\}$ by Lemma \ref{subset}.  Since for every $a \in (0,1)$ we have $ax \in D(x, a)$, it follows that $F(ax) \in D(F(x), a) = \{aF(x)\}$.
\vspace{3 mm}

(5) Suppose $F(x)$ is an extreme point of $B_Y$.  Then $D(F(x), 1/2)$ consists of one point, so (3) and the injectivity of $F$ imply that $D(x, 1/2)$ consists of one point, that is,  $x \in \ex(B_X)$.
\vspace{3 mm}

(6) Suppose that $F(x)$ is an extreme point of $B_Y$. By the surjectivity of $F$ there is a
$y \in S_X$ such that $F(y) = -F(x)$. Then $\|x - y\|_X \ge \|F(x) - (-F(x))\|_Y = 2$.
Let $z = \frac12 (x + y) \in B_X$. Since $\|x - z\|_X = \|y - z\|_X \le 1$, we have that $\|F(x) - F(z)\|_Y \le 1$
and $\|F(y) - F(z)\|_Y = \|-F(x) - F(z)\|_Y \le 1$. This means that $F(z) \in D_1(F(x)) =  \{0\}$, so $z = 0$ and $y = -x$.  Hence $F(-x) = -F(x)$.
\vspace{4 mm}

Now, let us proceed with the case when Y is strictly convex.
 \vspace{5 mm}

 (i) Assume to the contrary that $F(S_X) \neq S_Y$. By (2) this means that
$F(y) \not \in S_Y$ for some $y \in S_X$.  By the surjectivity of
$F$ there is an $x \in S_X$ such that $F(x) = F(y)/ \|F(y)\|_Y$. Then
from (4) it follows that $F(\|F(y)\|_Y x) = \|F(y)\|_Y F(x) = F(y)$, which contradicts the injectivity of $F$.
\vspace{3 mm}

Because of strict convexity of Y the items (ii) and (iii) are now evident consequences of (4) and (6) respectively.
\end{proof}

Following notations from \cite{CKOW2016} for every  $u \in S_X $ and  $v \in X$ denote by  $u^*(v)$  the directional derivative of the function $x \mapsto \|x\|_X$ at the point $u$ in the direction $v$:
$$
u^*(v) = \lim_{a \to 0^+}\frac1a \left(\|u + a v\|_X - \|u\|_X\right).
$$
By the convexity of the function $x \mapsto \|x\|_X$, the directional derivative exists.
If $E \subset X$ is a subspace and $u$ is a smooth point
of $S_E$ then $u^*|_E$ (the restriction of $u^*$ to $E$) is the unique norm-one
linear functional on $E$ that satisfies $u^*|_E(u) = 1$ (the supporting functional at point $u$). In general $u^*: X \to \R$ is not linear, but
it is sub-additive and positively homogeneous:
\begin{equation} \label{sublinear-1}
u^*(y_1) + u^*(y_2) \ge u^*(y_1 + y_2), \,\, u^*(ty) = tu^*(y), \ \text{for} \ t \ge 0.
\end{equation}
For proving these facts we will use the triangle inequality:
\begin{align*}
u^*(y_1 + y_2)  &=  \lim_{a \to 0^+}\frac1a \left(\|u + a (y_1+y_2)\|_X - \|u\|_X\right)=\\
&= \lim_{a \to 0^+}\frac{1}{2a} \left(2\|u + 2a (y_1+y_2)\|_X - 2\|u\|_X\right)=\\
&= \lim_{b \to 0^+}\frac{1}{b} \left(2\|u + b (y_1+y_2)\|_X - 2\|u\|_X\right)\le\\
&\le \lim_{b \to 0^+}\frac{1}{b} \left(\|u + b y_1\|_X - \|u\|_X\right)+\\
&+ \lim_{b \to 0^+}\frac{1}{b} \left(\|u + b y_2\|_X - \|u\|_X\right) = u^*(y_1) + u^*(y_2). \\
\end{align*}
\begin{align*}
tu^*(y)&= \lim_{a \to 0^+}\frac{t}{a} \left(\|u + ay\|_X - \|u\|_X\right) = \\
&=\lim_{b \to 0^+}\frac{1}{b} \left(\|u + bty\|_X - \|u\|_X\right)= u^*(ty).\\
\end{align*}
If one substitutes in the subadditivity condition $y_1 = v$ and $y_2 = -v$,
one gets
\begin{equation} \label{sublinear}
u^*(v)  \ge -u^*(-v).
\end{equation}
Also $u^*$  possesses the following property: for arbitrary $y_1, y_2 \in X$
\begin{equation} \label{x*}
u^*(y_1) - u^*(y_2) \le \|y_1 - y_2\|_X.
\end{equation}
Let us demonstrate the latter property:
\begin{align*}
u^*(y_1) - u^*(y_2)  &=  \lim_{a \to 0^+}\frac1a \left(\|u + a y_1\|_X - \|u + a y_2\|_X \right)\\
 &\le \lim_{a \to 0^+}\frac1a \left(\|u + a y_1 - u - a y_2\|_X \right) =  \|y_1 - y_2\|_X.
\end{align*}

The next lemma generalizes in a straightforward way \cite[Lemma 2.4]{CKOW2016}.
\begin{lem}\label{maj}
Let $F: B_X \to B_Y$ be a bijective non-expansive map, and suppose that  for some $u \in S_X $ and  $v \in B_X$ we have $u^*(-v)  = -u^*(v)$, $\|F(u)\| = \|u\|$ and $F(av) = a F(v)$ for all $a \in [-1,1]$.  Then $(F(u))^*(F(v)) = u^*(v)$.
\end{lem}
\begin{proof}
 \begin{align*}
(F(u))^*(-F(v)) &= \lim_{a \to 0^+}\frac1a \left(\|F(u) - a F(v)\|_Y - \|F(u)\|_Y\right)\\
                           &= \lim_{a \to 0^+}\frac1a \left(\|F(u) - F(av)\|_Y - \|u\|_X\right)\\
                           &\le \lim_{a \to 0^+}\frac1a \left(\|u - a v\|_X - \|u\|_X\right) = u^*(-v),
\end{align*}
so $(F(u))^*(-F(v)) \le  u^*(-v)$.  Substituting $-v$ in the place of $v$, we get also
$$
(F(u))^*(-F(-v)) \le  u^*(v).
$$
Using these two inequalities together with (\ref{sublinear}), we have
\bea
(F(u))^*(F(v)) &\ge& - (F(u))^*(-F(v)) \ge - u^*(-v) = u^*(v)  \\
&\ge& (F(u))^*(-F(-v)) =(F(u))^*(F(v)).
\eea
So all the inequalities in this chain are equalities.
\end{proof}

The following result and Corollary \ref{corollar111} are extracted from the proof of  \cite[Lemma 2.5]{CKOW2016}.

\begin{lem} \label{conv-smooth-prel}
  Let $F: B_X \to B_Y$ be a bijective non-expansive map such that $F(S_X) = S_Y$. Let $V \subset S_X$ be such a subset that $F(av) = a F(v)$ for all $a \in [-1,1]$, $v \in V$.  Denote $A = \{tx: x \in V, t \in [-1,1] \}$, then $F|_A$ is a bijective  isometry between $A$ and $F(A)$.
\end{lem}

\begin{proof}
Fix arbitrary $y_1, y_2 \in A$.   Let $E = \spn\{y_1, y_2\}$, and let $W \subset S_E$
be the set of smooth points of $S_E$ (which is dense in $S_E$). All the functionals $x^*$, where $x \in W$,  are linear on $E$, so $x^*(-y_i) = -x^*(y_i)$, for $i = 1,2$.  Also, according to our assumption, $F(ay_i) = a F(y_i)$ for all $a \in [-1,1]$. Now we can apply Lemma \ref{maj}.
\bea
 \|F(y_1) - F(y_2)\|_Y &\le& \|y_1 - y_2\|_X = \sup\{x^*(y_1 - y_2) : x \in W\} \\
 &=& \sup\{x^*(y_1) - x^*(y_2) : x \in W\}\\
&=&  \sup\{(F(x))^*(F(y_1)) - (F(x))^*(F(y_2)) : x \in W\} \\
&\le&
 \|F(y_1) - F(y_2)\|_Y,
\eea
where  on the last step we used the inequality \eqref{x*}. So $ \|F(y_1) - F(y_2)\| = \|y_1 - y_2\|$.
\end{proof}

\begin{corollary} \label{corollar111}
  If $F: B_X \to B_Y$ is a bijective non-expansive function that satisfies (i), (ii), and (iii) of Theorem \ref{wing-sc}, then $F$ is an isometry.
\end{corollary}
\begin{proof}
We can apply Lemma \ref{conv-smooth-prel} with $V = S_X$ and $A = B_X$.
\end{proof}


 \section{Main  results}

\begin{theo} \label{stric-conv-image}
  Let $F: B_X \to B_Y$ be a bijective non-expansive map. If $Y$ is  strictly convex,  then $F$ is an  isometry.
\end{theo}

\begin{proof}
If $Y$ is strictly convex, then $F$ satisfies (i), (ii), and (iii) of Theorem \ref{wing-sc}, so Corollary \ref{corollar111} is applicable.
\end{proof}

Our next goal is to show that each non-expansive bijection from the unit ball of arbitrary Banach space to the unit ball of $\ell_1$ is an isometry. In the proof we will use the following three known results.

\begin{prop}[P.~Mankiewicz's  \cite{mank}] \label{Mankiewicz}
If $A\subset X$ and $B \subset Y$ are convex with non-empty interior, then every bijective isometry $F : A \to B$ can be extended to a bijective affine isometry $\tilde F : X \to Y $.
\end{prop}
Taking into account that in the case of $A$, $B$ being the unit balls every isometry  maps 0 to 0, this result implies that every bijective isometry $F : B_X \to B_Y$  is the restriction of a linear isometry from $X$ onto $Y$.

\begin{prop} [Brower's invariance of domain principle \cite{Brouwer1912}]  \label{Brower}
Let $U$ be an open subset of $\R^n$ and $f : U \to \R^n$ be an injective continuous map, then $f(U)$ is open in $\R^n$.
\end{prop}

\begin{prop}[Proposition 4 of \cite{KZ}] \label{prop-surject}
Let $X$ be a  finite-dimensional normed space and $V$ be a subset of $B_X$ with the following two properties:  $V$ is homeomorphic to $B_X$ and $V \supset S_X$. Then $V=B_X$.
\end{prop}

Now we give the promised theorem.

\begin{theo} \label{conv-smooth}
  Let $X$ be a Banach space, $F: B_X \to B_{\ell_1}$ be a bijective non-expansive map. Then $F$ is an  isometry.
\end{theo}
\begin{proof} Denote  $e_n = (\delta_{i,n})_{i \in \N}$, $n = 1,2, \ldots$ the elements of the canonical basis of $\ell_1$ (here, as usual, $\delta_{i,n} = 0$ for $n \neq i$ and $\delta_{n,n} = 1$). It is well-known and easy to check that $ext(B_{\ell_1}) = \{ \pm e_n,  i = 1,2,...\}$.

Denote $g_n = F^{-1}e_n$. According to item (5) of Theorem \ref{wing-sc} each of $g_n$ is an extreme point of $B_X$.

One more notation: for every  $N \in \N$ and $X_N = \spn\{g_k\}_{k \le N}$ denote $U_N$ and $\partial U_N$ the unit ball and the unit sphere of $X_N$ respectively and analogously for $Y_N = \spn\{e_k\}_{k \le N}$ denote $V_N$ and $\partial V_N$ the unit ball and the unit sphere of $Y_N$ respectively.

\textbf{Claim}. \textit{For every  $N \in \N$ and every collection $\{a_k\}_{k \le N}$ of reals with $\|\sum_{n \le N}a_n g_n\| \le 1$}
$$
F\left(\sum_{n \le N}a_n g_n\right) = \sum_{n \le N}a_n e_n.
$$

{\it Proof of the Claim}. We will use induction on $N$. For $N = 1$, the Claim follows from items (4) and (6) of Theorem \ref{wing-sc}. Now assume the validity of  the Claim for $N-1$, and let us prove it for $N$.
At first, for every $x = \sum_{i=1}^{N}{\alpha_i  g_i}$ we will show that
\begin{equation} \label{eq-norm-l1}
\|x\| = \sum_{i=1}^{N}{|\alpha_i|}.
\end{equation}
  Note that, due to the positive homogeneity of the norm, it is sufficient to consider  $x = \sum_{i=1}^{N}{\alpha_i  g_i}$, with  $\sum_{i=1}^{N}{|\alpha_i|}\leq 1$. In such a case $x \in U_N$. So
$$
\left\|\sum_{i=1}^{N-1}{\alpha_i  g_i}\right\| \leq \sum_{i=1}^{N-1}{\|\alpha_i  g_i\|} = \sum_{i=1}^{N-1}{|\alpha_i|}\leq \sum_{i=1}^{N}{|\alpha_i|}\leq 1,
$$
and $\sum_{i=1}^{N-1}{\alpha_i  g_i }\in U_N$. On one hand,
$$
\|x\| = \left\|\sum_{i=1}^{N}{\alpha_i  g_i}\right\| \leq \sum_{i=1}^{N}{|\alpha_i|}.
$$
 On the other hand, by the induction hypothesis $F(\sum_{i=1}^{N-1}{\alpha_i  g_i}) =  \sum_{i=1}^{N-1}{\alpha_i  e_i}$ and by items (4) and (6) of Theorem \ref{wing-sc} $F(-\alpha_N  g_N) = -\alpha_N  e_N$. Consequently,
\bea
\|x\| &=& \left\|\sum_{i=1}^{N-1}{\alpha_i  g_i}+ \alpha_N  g_N\right\|
= \rho\left(\sum_{i=1}^{N-1}{\alpha_i  g_i}, (-\alpha_N  g_N)\right) \\
&\geq&  \rho\left(F(\sum_{i=1}^{N-1}{\alpha_i  g_i}), F(-\alpha_N  g_N)\right)
= \left\| \sum_{i=1}^{N-1}{\alpha_i  e_i}+ \alpha_N  e_N\right\| = \sum_{i=1}^{N}{|\alpha_i|},
\eea
and \eqref{eq-norm-l1} is demonstrated.
That means that
$$
U_N = \left\{\sum_{n \le N}a_n g_n: \sum_{n  \le N} |a_n|  \le 1 \right\},  \partial U_N = \left\{\sum_{n  \le N}a_n g_n: \sum_{n  \le N} |a_n|  = 1 \right\}.
$$

The remaining part of the proof of the Claim, and of the whole theorem repeats almost literally the corresponding part of the proof of \cite[Thorem 1]{KZ}, so we present it here only for the reader's convenience. \\
 Let us show that
\begin{equation} \label{eq1}
F(U_N) \subset V_N.
\end{equation}
To this end, consider $x \in U_N$. If $x$ is of the form $\alpha g_N$ the statement follows from Theorem \ref{wing-sc}. So we must consider $x = \sum_{i=1}^{N}{\alpha_i  g_i}$, $\sum_{i=1}^{N}{|\alpha_i|}\leq 1$ with $\sum_{i=1}^{N - 1}|\alpha_i| \neq 0$. Denote the expansion of $F(x)$ by $F(x) = \sum_{i=1}^{\infty}{y_i  e_i}$. For the element
$$
x_1 = \frac{ \sum_{i=1}^{N - 1}\alpha_i  g_i}{ \sum_{i=1}^{N - 1}|\alpha_i| }
$$
by the induction hypothesis
$$
F(x_1) = \frac{ \sum_{i=1}^{N - 1}\alpha_i  e_i}{ \sum_{i=1}^{N - 1}|\alpha_i|}.
$$

So we may write the following chain of inequalities:

\begin{align*}
2 &=  \left\|F(x_1) - \frac{\alpha_N}{|\alpha_N|} e_N\right\|  \le \left\|F(x_1) - \sum_{i=1}^{N}{y_i  e_i}\right\| + \left\|\sum_{i=1}^{N}{y_i  e_i} - \frac{\alpha_N}{|\alpha_N|} e_N\right\|  \\
& = \left\|F(x_1) - F(x)\right\| + \left\|F(x) - \frac{\alpha_N}{|\alpha_N|} e_N\right\|  - 2 \sum_{i=N+1}^{\infty}|y_i| \\
&\le
 \left\|F(x_1) - F(x)\right\| + \left\|F(x) - F\left(\frac{\alpha_N}{|\alpha_N|}g_N\right)\right\|    \le \left\|x_1 - x \right\| + \left\|x - \frac{\alpha_N}{|\alpha_N|}g_N\right\| \\
 &= \sum_{j=1}^{N - 1}\left| \alpha_j -  \frac{\alpha_j}{ \sum_{i=1}^{N - 1}|\alpha_i|} \right| + |\alpha_N| +   \sum_{j=1}^{N - 1}| \alpha_j| + \left|\alpha_N - \frac{\alpha_N}{|\alpha_N|}\right|\\
 &= \sum_{j=1}^{N - 1}| \alpha_j| \left(1 + \left| 1 -  \frac{1}{ \sum_{i=1}^{N - 1}|\alpha_i|} \right| \right) +  |\alpha_N|\left(1 + \left| 1 -  \frac{1}{|\alpha_N|} \right| \right) = 2.
\end{align*}
This means that all the inequalities in between are in fact equalities, so in particular $\sum_{i=N+1}^{\infty}|y_i| = 0$, i.e. $F(x) = \sum_{i=1}^{N}{y_i  e_i} \in V_N$ and \eqref{eq1} is proved.

Now, let us demonstrate that
\begin{equation} \label{eq2}
F(U_N) \supset \partial V_N.
\end{equation}

Assume on the contrary that there is a $y \in \partial V_N \setminus F(U_N)$. Denote $x = F^{-1}(y)$. Then,  $\|x\| = 1$ (by (2) of Theorem \ref{wing-sc}) and $x \notin U_N$. For every $ t \in [0, 1]$ consider $F(tx)$. Let $F(tx) = \sum_{n \in \N} b_n e_n$ be the corresponding expansion. Then,
\begin{align*}
1 &= \|0 - tx\| + \|tx - x\| \ge \|0 - F(tx)\| + \|F(tx) - y\|  \\
& = 2\sum_{n > N} |b_n| + \left \| \sum_{n \le N}b_n e_n \right\| + \left\|y - \sum_{n \le N}b_n e_n \right\| \ge 2\sum_{n > N} |b_n| + 1,
\end{align*}
so $\sum_{n > N} |b_n| = 0$. This means that $F(tx) \in V_N$ for every $ t \in [0, 1]$. On the other hand,  $F(U_N)$ contains a relative neighborhood of 0 in $V_N$ (here we use that $F(0) = 0$ and Proposition \ref{Brower}), so the continuous curve $\{F(tx) : t \in [0, 1]\}$   in  $V_N$ which connects 0 and $y$ has a non-trivial intersection with  $F(U_N)$.  This implies that there is a $t \in (0, 1)$ such that $F(tx) \in F(U_N)$. Since $tx \notin U_N$ this contradicts the injectivity of $F$. Inclusion \eqref{eq2} is proved.

 Now, inclusions  \eqref{eq1} and \eqref{eq2} together with Proposition \ref{prop-surject} imply $F(U_N) = V_N$. Observe that by \eqref{eq-norm-l1} $U_N$ is isometric to $V_N$ and, by finite dimensionality, $U_N$ and $V_N$ are compacts. So, $U_N$ and $V_N$ can be considered as two copies of one the same compact metric space, and  Theorem 1.1 of  \cite{NaiPioWing} implies that every bijective non-expansive map from $U_N$ onto $V_N$  is an isometry. In particular, $F$ maps $U_N$ onto $V_N$ isometrically. Finally, the application of Proposition \ref{Mankiewicz} gives us that the restriction of $F$ to $U_N$ extends to a linear map from $X_N$ to $Y_N$, {\it which completes the proof of the Claim}.

Now let us complete the proof of the theorem. At first, passing in \eqref{eq-norm-l1} to limit as $N \to \infty$ we get
 $$
\|z\| = \sum_{i=1}^{\infty}{|z_i|}
$$
 for every $z = \sum_{n =1}^\infty z_n g_n$ with $\sum_{n =1}^\infty |z_n| < \infty$. The continuity of $F$ and the claim imply that for every $x = \sum_{n =1}^\infty x_n e_n \in B_{\ell_1}$
 $$
F\left(\sum_{n =1}^\infty x_n g_n\right) = \sum_{n =1}^\infty x_n e_n, \quad\text{so} \quad F^{-1}\left(\sum_{n =1}^\infty x_n e_n\right) = \sum_{n =1}^\infty x_n g_n.
$$
Consequently, for every  $x,y \in B_{\ell_1}$, $x =\sum_{n =1}^\infty x_n e_n$ and $y =\sum_{n =1}^\infty y_n e_n $ the following equalities  hold true:
\begin{align*}
\|x-y\| &=
\left\| \sum_{n =1}^\infty (x_n - y_n) e_n\right\| = \sum_{n =1}^\infty |(x_n - y_n)| =\left\| \sum_{n =1}^\infty (x_n - y_n) g_n\right\| \\
&=\left\| \sum_{n =1}^\infty x_n g_n - \sum_{n =1}^\infty y_n g_n\right\| = \left\| F^{-1}(x) - F^{-1}(y) \right\|.
\end{align*}
So, $F^{-1}$ is an isometry, consequently the same is true for $F$.
\end{proof}

Our next (and the last) goal is to demonstrate that each non-expansive bijection between unit balls of  two different finite dimensional Banach spaces is an isometry.  Below we recall the definitions and well-known properties of total and norming subsets of dual spaces that  we will need further.

A subset $V \subset S_{X^*}$  is called \emph{total} if for every $x\neq 0$ there exists $f\in V$ such that $f(x)\neq 0$. $V $  is called norming if $\sup|f(x)|_{f\in V}=\|x\|$ for all $x\in X$. We will use the following easy exercise.

\begin{lem}[\cite{Kad}, Exercise 9, p. 538] \label{norm}
Let $A\subset S_X$  be dense in $S_X$, for every $a\in A$ let  $f_a$ be  a supporting functional at $a$. Then $ V = \{f_a: a\in A\}$ is norming (and consequently total).
\end{lem}

The following known fact is an easy consequence of the bipolar theorem.
\begin{lem} \label{refl}
Let $X$ be a reflexive space. Then $V\subset S_{X^*}$ is norming if and only if $\overline{\aconv}(V) = {B_{X^*}}$.
\end{lem}

Now we can demonstrate the promised result.

\begin{theo} \label{Fin-dim}
 Let $X$, $Y$ be Banach spaces, $Y$ be finite-dimensional, $F:B_X \to B_Y$ be a bijective non-expansive map. Then $F$ is an isometry.
 \end{theo}
 \begin{proof}
Take an arbitrary finite-dimensional subspace $Z \subset X$. Then the restriction of $F$ to $B_Z$  is a bijective and continuous map between two compact sets $B_Z$ and $F(B_Z)$, so $B_Z$ and $F(B_Z)$ are homeomorphic. Thus,  Brower's invariance of domain principle (Proposition \ref{Brower})  implies that $\dim Z \le \dim Y$. By arbitrariness of  $Z \subset X$ this implies that $\dim X \le \dim Y$. Consequently, $F$ being bijective and continuous map between compact sets $B_X$ and $B_Y$, is a homeomorphism. Another application of  Proposition \ref{Brower} says that $\dim X = \dim Y$, $F$ maps interior points in  interior points, and $F(S_X) = S_Y$.

Let $G$ be the set of all $x \in S_X$ such that norm is differentiable both at $x$ and $F(x)$.  According to \cite[Theorem 25.5]{Rock},  the complement to the set of differentiability points of the norm is meager. Consequently, $G$ being an intersection of two comeager sets,   is dense in $S_X$.
Recall that $F$ is a homeomorphism, so $F(G)$ is dense in $S_Y$. Thus, Lemma \ref{norm} ensures that   $A:= \{x^*: x\in G\}$ and  $B:= \{F(x)^*:x\in G\} = \{y^* :y\in F(G)\}$ are norming subsets of $X^*$ and $Y^*$ respectively, and consequently by Lemma \ref{refl}
\begin{equation} \label{eq:AB}
\overline{\aconv}(A) = B_{X^*},  \overline{\aconv}(B) = B_{Y^*}
\end{equation}

    Denote $K = F^{-1}(\ex B_Y) \subset \ex B_X$. Note that  for all $x \in G$ the corresponding $(F(x))^*$ and $x^*$ are linear, and Lemma \ref{maj} implies that for all $x \in G$ and $z \in K$ the following equality holds true:
\begin{equation*}\label{eq*}
(F(x))^*(F(z))=x^*(z).
\end{equation*}

Let us define the map $H:A\to B$ such that $H(x^*)=(F(x))^*$.
For the correctness of this definition it is necessary to verify for all $x_1, x_2 \in G$  the implication
$$
({x_1}^*={x_2}^*) \Longrightarrow ({F(x_1)}^*={F(x_2)}^*).
$$
Assume for given $x_1, x_2 \in G$ that ${x_1}^*={x_2}^*$.
  In order to check equality ${F(x_1)}^*={F(x_2)}^*$ it is sufficient to verify that ${F(x_1)}^*y ={F(x_2)}^*y$  for $y \in \ex B_Y$, i.e. for $y$ of the form $y = F(x)$ with $x \in K$.
Indeed,
 $$
 {F(x_1)}^*(F(x))={x_1}^*(x)={x_2}^*(x)={F(x_2)}^*(F(x)).
 $$
Let us extend $H$ by linearity to $\tilde H: X^* = \spn(x^*, x\in G) \to Y^*$. For $x^*=\sum_{k=1}^N\lambda_k {x_k}^*$, $x_k \in G$ let $\tilde H(x^*) = \sum_{k=1}^N\lambda_k H({x_k}^*)$.
To verify the correctness of this extension we will prove that
 $$
\left( \sum_{k=1}^N\lambda_k {x_k}^*=\sum_{k=1}^M\mu_k {y_k}^*\right)  \Longrightarrow \left(\sum_{k=1}^N\lambda_k H({x_k}^*)=\sum_{k=1}^M\mu_k H({y_k}^*)\right).
$$
 Again we will prove equality $\sum_{k=1}^N\lambda_k H({x_k}^*)=\sum_{k=1}^M\mu_k H({y_k}^*)$ of functionals only on elements of the form $y = F(x)$ with $x \in K$.
  $$
 \left(\sum_{k=1}^N\lambda_k H({x_k}^*)\right) F(x)=\sum_{k=1}^N\lambda_k {F(x_k)}^*(F(x))=\sum_{k=1}^N\lambda_k {x_k}^*(x) $$
  $$ = \sum_{k=1}^M\mu_k {y_k}^*(x)= \sum_{k=1}^M\mu_k {F(y_k)}^*(F(x))= \left(\sum_{k=1}^M\mu_k H({y_k}^*)\right)F(x).
  $$
Observe that, according to \eqref{eq:AB}, $\tilde H (X^*) = \spn H(A) = \spn B = Y^*$, so
$\tilde H$ is surjective, and consequently, by equality of corresponding dimensions, is bijective. Recall, that $\tilde H(A) =  H(A) = B$, so $\tilde H$ maps $A$ to $B$ bijectively. Applying again \eqref{eq:AB} we deduce
 that $\tilde H(B_{X^*}) = B_{Y^*}$ and $X^*$ is isometric to $Y^*$.  Passing to the duals we deduce that $Y^{**}$ is isometric to  $X^{**}$  (with $\tilde H^*$ being the corresponding isometry), that is $X$ and $Y$ are isometric. So, $B_X$ and $B_Y$ are two copies of the same compact metric space, and the application of  EC-plasticity of compacts \cite[Theorem 1.1]{NaiPioWing} completes the proof.
\end{proof}

\vspace{2mm}
\noindent {\bf Acknowledgement.}
The author is grateful to her scientific advisor Vladimir Kadets for constant help with this project.

\end{document}